\documentclass{amsart}
\usepackage[margin=.99in]{geometry}
\usepackage{amssymb,latexsym}
\usepackage[dvipsnames]{xcolor}
\usepackage[parfill]{parskip}
\usepackage[all]{xy}
\usepackage{ulem}

\usepackage[colorlinks]{hyperref}
\hypersetup{
    colorlinks,
    linkcolor={ForestGreen},
    citecolor={MidnightBlue},
    urlcolor={Black}
}

\theoremstyle{plain}
\newtheorem{theorem}{Theorem}[section]
\newtheorem{corollary}[theorem]{Corollary}
\newtheorem{proposition}[theorem]{Proposition}
\newtheorem{lemma}[theorem]{Lemma}
\theoremstyle{definition}
\newtheorem{definition}[theorem]{Definition}

\newtheorem{remark}[theorem]{Remark}

\newtheorem{conjecture}[theorem]{Conjecture}

\newcommand{\remarkend}{\,\hfill$\clubsuit$}

\newcommand{\enm}[1]{\ensuremath{#1}}

\renewcommand{\bar}[1]{\overline{#1}}

\newcommand{\ZZ}{\enm{\mathbb{Z}}}

\newcommand{\PP}{\enm{\mathbb{P}}}

\renewcommand{\phi}{\varphi}
\renewcommand{\theta}{\vartheta}
\renewcommand{\epsilon}{\varepsilon}

\newcommand{\calI}{\mathcal{I}}

\newcommand{\calO}{\mathcal{O}}

\newcommand{\calS}{\mathcal{S}}

\newcommand{\bbP}{\mathbb{P}}

\newcommand{\bbR}{\mathbb{R}}

\DeclareMathOperator{\CI}{\mathrm{CI}}
\DeclareMathOperator{\HF}{\mathrm{HF}}

\DeclareMathOperator{\Hilb}{\mathrm{Hilb}}

\DeclareMathOperator{\red}{red}

\DeclareMathOperator{\codim}{\mathrm{codim}}
\DeclareMathOperator{\coeff}{\mathrm{coeff}}
\DeclareMathOperator{\str}{\mathrm{str}}
\DeclareMathOperator{\slrk}{\mathrm{sl.rk}}

\date{}

\title{Strength and slice rank of forms are generically equal}

\author{Edoardo Ballico}
\author{Arthur Bik}
\author{Alessandro Oneto}
\author{Emanuele Ventura}

\address[E. Ballico, A. Oneto]{Universita di Trento, Via Sommarive, 14 - 38123 Povo (Trento), Italy}
\email{edoardo.ballico@unitn.it, alessandro.oneto@unitn.it}

\address[A. Bik]{MPI for Mathematics in the Sciences, Leipzig, Germany}
\email{arthur.bik@mis.mpg.de}

\address[E. Ventura]{Universit\"at Bern, Mathematisches Institut, Sidlerstrasse 5, 3012 Bern, Switzerland}
\email{emanueleventura.sw@gmail.com, emanuele.ventura@math.unibe.ch}

\begin{document}
\begin{abstract}
We prove that strength and slice rank of homogeneous polynomials of degree $d \geq 5$ over an algebraically closed field of characteristic zero coincide generically. To show this, we establish a conjecture of Catalisano, Geramita, Gimigliano, Harbourne, Migliore, Nagel and Shin concerning dimensions of secant varieties of the varieties of reducible homogeneous polynomials. These statements were already known in degrees $2\leq d\leq 7$ and $d=9$.
\end{abstract}
\maketitle

\section{Introduction}

Ananyan and Hochster \cite{ah} introduced the notion of {\it strength} of a polynomial to solve a famous conjecture by Stillman on the existence of a uniform bound, independent on the number of variables, for the projective dimension of a homogeneous ideal of a polynomial ring. Recently, polynomial strength and related questions have been intensively investigated \cite{ah2, bbov, bv20, bde, des, ess, kz}. 

Let $\Bbbk$ be an algebraically closed field of characteristic zero, let $n\geq 1$ be an integer and let
\[
\calS = {\textstyle \bigoplus_{d\geq 0}} \calS_d := \Bbbk[x_0,\ldots,x_n]
\]
be the standard graded polynomial ring in $n+1$ variables over $\Bbbk$. So the elements of $\calS_d$ are homogeneous polynomials, also called {\it forms}, of degree $d$. Fix an integer $d\geq 2$ and let $f\in\calS_d$ be a degree-$d$ form. 

\begin{definition}
The \textit{strength} of $f$ is the minimal integer $r\geq0$ for which there exists a decomposition
\[
f = g_1\cdot h_1+\ldots + g_r\cdot h_r
\]
where $g_1,h_1,\ldots,g_r,h_r$ are forms of positive degree. We denote it by $\str(f)$.
\end{definition}

Computing the strength of a given polynomial is a very difficult task. Hence, a natural problem is to determine the strength of a general homogeneous polynomial. In \cite{BO}, A.B. and A.O. noticed that a conjectural answer to this problem was implicitly given in \cite[Remark 7.7]{cgghmns} where the authors study dimensions of secant varieties of the varieties of reducible forms. In particular, it was conjectured that the strength of a general form coincides with its slice rank; see \cite[Conjecture 1.1]{BO}. Recall that the value of the slice rank of a general form is classically known; see Remark \ref{rmk:general_slice}.

\begin{definition}
The \textit{slice rank} of $f$ is the minimal integer $r\geq0$ for which there exists a decomposition
\[
f = \ell_1\cdot h_1+\ldots + \ell_r\cdot h_r
\]
where $\ell_1,\ldots,\ell_r$ are linear forms and $h_1,\ldots,h_r$ are forms of degree $d-1$. We denote it by $\slrk(f)$.
\end{definition}
\begin{conjecture}[{\cite[Conjecture 1.1]{BO}}]\label{conj:strength} 
The strength and the slice rank of a general form in $\calS_d$ are equal.
\end{conjecture}

So far, this conjecture has been established in the following cases:  when the degree $d$ is larger than $\frac{3}{2}n+\frac{1}{2}$ \cite{Sz96:CompleteIntersections}, when twice the general slice rank is at most $n+2$ \cite{ccg} and for $d \leq 7$ and $d = 9$ \cite{BO}.

The aim of this paper is to establish  Conjecture~\ref{conj:strength}, thereby determining the strength of a general form, by proving the stronger conjecture from \cite[Remark 7.7]{cgghmns} which we also state below. 

\subsection*{Geometric formulation of the problem.}

For an integer $1\leq j\leq d/2$, we consider the \textit{variety of forms with a degree-$j$ factor} $X_j := \{[g\cdot h] \mid g\in\bbP\calS_j, h\in\bbP\calS_{d-j}\} \subseteq \bbP\calS_d$. The union of these varieties is the \textit{variety of reducible forms} $X_{\red} := {\textstyle \bigcup_{j=1}^{\left\lfloor d/2 \right\rfloor}} X_j$. For an integer $r\geq 1$, the $r$th \textit{secant variety} of $X_{\red}$ is the Zariski-closure
\[
\sigma_r(X_{\red}) := \overline{\left\{ [f] \in \bbP\calS_d \,\middle|\, f = f_1 +\ldots+ f_r,~ [f_1],\ldots,[f_r] \in X_{\red}\right\}}
\]
of the union of all linear spaces spanned by $r$ points on $X_{\rm red}$. Since $X_{\rm red}$ is reducible, we can describe its $r$th secant variety as
\[
\sigma_r(X_{\red}) = \bigcup_{1\leq a_1,\ldots,a_r \leq \left\lfloor d/2 \right\rfloor} J_{a_1,\ldots,a_r}\vspace{-6pt}
\]
where
\[
J_{a_1,\ldots,a_r} := J(X_{a_1},\ldots,X_{a_r}) = \overline{\left\{ [f] \in \bbP\calS_d \,\middle|\, f =f_1 +\ldots+ f_r,~ [f_1]\in X_{a_1},\ldots,[f_r]\in X_{a_r}\right\}}
\]
is the \textit{join} of the varieties $X_{a_1},\ldots,X_{a_r}$. Now, the general slice rank and strength are
\[
\slrk_{d,n}^\circ := \min\{r\in\ZZ_{\geq0} \mid \sigma_r(X_1) = \bbP\calS_d\} \quad \text{ and } \quad \str_{d,n}^\circ := \min\{r\in\ZZ_{\geq0} \mid \sigma_r(X_{\red}) = \bbP\calS_d\}.
\]
So Conjecture \ref{conj:strength} is implied by the following stronger conjecture.

\begin{conjecture}[{\cite[Remark 7.7]{cgghmns}}]\label{conj:dimensions}
For each integer $r\geq1$, we have $\dim \sigma_r(X_{\red}) = \dim \sigma_r(X_1)$.
\end{conjecture}

\begin{remark}\label{rmk:general_slice}
Recall that the value of the general slice rank is classically known as it equals the minimal codimension of a linear space contained in a general hypersurface. If $d\geq 3$, then we have 
\[
\slrk^\circ_{d,n} := \min\left\{ r\in\ZZ_{\geq0} \,\middle|\, r(n+1-r) \geq \binom{n-r+d}{d}\right\}\vspace{-6pt}
\]
and 
\[
\codim_{\PP\calS_d}\sigma_r(X_1)=\binom{n-r+d}{d}-r(n+1-r)
\]
for all integers $1\leq r<\slrk^\circ_{d,n}$ by \cite[Theorem~12.8]{Harris:First}. Note that $\slrk^\circ_{d,n} \leq n$. So we can (and often will) relax the assumption $r < \slrk_{d,n}^\circ$ to $r < n$.
\remarkend
\end{remark}

The classical approach to computing dimensions of secant and join varieties is via \textit{Terracini's Lemma} \cite{ter} which asserts that, if $Y_1,\ldots,Y_r \subseteq \bbP^N$ are projective varieties, $q_1\in Y_1,\ldots, q_r\in Y_r$ are general points and $p \in \langle q_1,\ldots,q_r\rangle$ is general, then
\[
T_p\sigma_r J(Y_1,\ldots,Y_r) = \langle T_{q_1}Y_1,\ldots,T_{q_r}Y_r \rangle;
\]
see e.g. \cite[Lemma 1]{guide} for a recent presentation. By direct computation, it is easy to observe that the tangent space to $X_a$ at a general point $[g\cdot h]$, with $\deg(g) = a$ and $\deg(h) = d-a$, is given by $\bbP(g,h)_d$ where $(g,h)_d := (g,h) \cap \calS_d$ is the degree-$d$ homogeneous part of the ideal generated by $g$ and $h$. Therefore
\begin{equation}\label{eq:tangent_join}
\dim J_{a_1,\ldots,a_r} = \dim (g_1,h_1,\ldots,g_r,h_r)_d -1,
\end{equation}
where $g_i,h_i$ are general forms with $\deg(g_i) = a_i$ and $\deg(h_i) = d-a_i$. The codimensions of the homogeneous parts of a homogeneous ideal are encoded in its \textit{Hilbert function}, whose generating power series is called the \textit{Hilbert series}. These are among the most studied algebraic invariants of a homogeneous ideal. The Hilbert series of an ideal generated by general forms is prescribed by Fr\"oberg's famous conjecture; see~\cite{fro}. In \cite[Theorem 5.1]{ccg}, the authors used the known cases of Fr\"oberg's conjecture to deduce the integers $d,n,r,a_1,\ldots,a_r$ with $2r \leq n+2$ for which $J_{a_1,\ldots,a_r} = \bbP\calS_d$. Similarly, in \cite[Theorem~7.4]{cgghmns}, the authors showed that Conjecture \ref{conj:dimensions} holds if $2r \leq n+1$. The strength of the general form corresponds to the minimal codimension of a complete intersection inside a general hypersurface. This is the perspective of \cite[Corollary A]{Sz96:CompleteIntersections}, where the author shows that Conjecture~\ref{conj:strength} holds if $d \geq \frac{3}{2}n+\frac{1}{2}$. 

In \cite{BO}, A.B. and A.O. proved the following results.

\begin{theorem}
Let $d\in\{3,4,5,6,7,9\}$ and $n,r\geq 1$ be integers such that $r<\slrk_{d,n}^\circ$. Then Conjecture \ref{conj:dimensions} holds. Furthermore, unless $(d,n,r) = (4,3,2)$, the subvariety $\sigma_r(X_1)$ is the unique component of $\sigma_r(X_{\red})$ of maximal dimension. If $(d,n,r) = (4,3,2)$, the codimensions of $\sigma_r(X_1)$, $J(X_1,X_2)$ and $\sigma_r(X_2)$ each equal~$1$.
\end{theorem}

\begin{corollary}
When $d \leq 7$ and $d=9$, the general form of $\calS_d$ has strength equal to its slice rank.
\end{corollary}

\begin{samepage}
The main results of this paper are the following complementing theorem and corollary.

\begin{theorem}\label{thm:main}
Let $d\geq 5$ and $n,r\geq 1$ be integers such that $r<\slrk_{d,n}^\circ$. Then Conjecture \ref{conj:dimensions} holds. Furthermore, the subvariety $\sigma_r(X_1)$ is the unique component of $\sigma_r(X_{\red})$ of maximal dimension.
\end{theorem}
\end{samepage}

\begin{corollary}
The general form of $\calS_d$ has strength equal to its slice rank.
\end{corollary}

\subsection*{Structure of the paper}

In Section \ref{sec:upperbound}, we find a numerical upper bound for the dimension of $J_{a_1,\ldots,a_r}$, which is an equality for $a_1=\ldots= a_r=1$ and $r<\slrk_{d,n}^{\circ}$. In Section \ref{sec:proof}, we study this upper bound as $a_1,\ldots,a_r$ vary and prove the main result.

\subsection*{Acknowledgements}
E.B. is partially supported by MIUR and GNSAGA of INdAM (Italy). E.V. is supported by Vici Grant 639.033.514 of Jan Draisma from the  Netherlands Organisation for Scientific Research.

\section{An upper bound on the dimensions}\label{sec:upperbound}

Let $d,n\geq 2$, $r<n$ and $a_1,\ldots,a_r \leq d/2$ be positive integers. We consider the subset  
\[
J^\circ_{a_1,\ldots,a_r} := \left\{ f\in\bbP\calS_d \,\middle|\, f = \sum_{i=1}^r g_i\cdot h_i, \quad g_i \in \calS_{a_i}, h_i \in \calS_{d-a_i}, \quad (g_1,\ldots,g_r)\text{ is a complete intersection}\right\}
\]
of $J_{a_1,\ldots,a_r}$. Let $\CI_n(a_1,\ldots,a_r)$ be the set of complete intersections in $\bbP^n$ of codimension $r$ defined by the intersection of hypersurfaces of degrees $a_1,\ldots,a_r$. 

In order to give an upper bound on the dimension of $J_{a_1,\ldots,a_r}$, we first observe that the subset $J^\circ_{a_1,\ldots,a_r}$ is dense and then we bound the dimension of this subset by parametrizing it via the space of complete intersections $\CI_n(a_1,\ldots,a_n)$ whose dimension can be computed explicitely.

\begin{samepage}
\begin{lemma}\label{lemma:dense}
The subset $J^\circ_{a_1,\ldots,a_r}$ is dense in $J_{a_1,\ldots,a_r}$.
\end{lemma}
\begin{proof}
Let $[f]\in\bbP\calS_d$ be a form which admits a strength decomposition $f = \sum_{i=1}^r g_ih_i$ with $\deg(g_i) = a_i$. It is enough to show that $f\in\overline{J^\circ_{a_1,\ldots,a_r}}$. Consider general forms $(u_1,\ldots,u_r) \in \calS_{a_1}\times\cdots\times\calS_{a_r}$. By generality, since $r \leq n$, the $u_i$'s form a regular sequence. Since being a regular sequence is an open condition in the Zariski topology, there exists an $\epsilon > 0$ such that
\[
(su_1+ g_1,\ldots, su_r+ g_r)\in \calS_{a_1}\times\cdots\times\calS_{a_r} \quad \text{ is a regular sequence for all } s \in (0,\epsilon]\cap\mathbb{Q}.
\]
For $s\in (0,\epsilon]\cap\mathbb{Q}$, define $f_s := \sum_{i=1}^r (g_i+su_i)h_i\in J^{\circ}_{a_1,\ldots,a_r}$. Then $\lim_{s\to0} f_s = f$ and hence $f\in\overline{J^\circ_{a_1,\ldots,a_r}}$. 
\end{proof}
\end{samepage}

\begin{lemma}\label{lemma:bound1}
We have $\dim J_{a_1,\ldots,a_r} \leq \dim \CI_n(a_1,\ldots,a_r) + \binom{n+d}{d} - \coeff_d\left(\frac{\prod_{i=1}^r (1-t^{a_i})}{(1-t)^{n+1}}\right)-1$.
\end{lemma}
\begin{proof}
If $I = (g_1,\ldots,g_r) \subseteq \calS$ is an ideal defined by a regular sequence of degrees $a_1,\ldots,a_r$, then
\[
\dim\,(\calS/I)_d = \coeff_d\left(\frac{\prod_{i=1}^r (1-t^{a_i})}{(1-t)^{n+1}}\right). 
\]
Hence
\[
\dim\,(g_1,\ldots,g_r)_d = \binom{n+d}{d} - \coeff_d\left(\frac{\prod_{i=1}^r (1-t^{a_i})}{(1-t)^{n+1}}\right) =: N + 1.
\]
From Lemma~\ref{lemma:dense}, we derive that $\dim J^{\circ}_{a_1,\ldots,a_r} = \dim J_{a_1,\ldots, a_r}$. 
Now, let $E$ be the projective bundle on $\CI_n(a_1,\ldots,a_r)$ whose fiber at a point $Y\in \CI_n(a_1,\ldots,a_r)$ is the projective space $\bbP(I_Y)_d\cong \bbP^N$. Then 
\[
\dim E = \dim \CI_n(a_1,\ldots,a_r) +N.
\]
We consider the morphism $E \longrightarrow J^{\circ}_{a_1,\ldots,a_r}$ given by $\left(Y, f\right)\mapsto f$. This map is surjective by definition of $E$ and $J^{\circ}_{a_1,\ldots,a_r}$. Thus
\[
\dim J_{a_1,\ldots,a_r} = \dim J^{\circ}_{a_1,\ldots,a_r}  \leq \dim E = \dim \CI_n(a_1,\ldots,a_r) +N,
\]
which gives the desired upper bound. 
\end{proof}

\begin{samepage}
Now, we compute the dimension of $\CI_n(a_1,\ldots,a_r)$. 

\begin{remark}\label{rmk:deformation}
The Hilbert polynomial $P_{a_1,\ldots,a_r}(t)$ of a complete intersection is uniquely determined by the degrees defining it since it is computed from the Koszul complex. In \cite[Section 4.6.1]{ser}, it is shown that $\CI_n(a_1,\ldots,a_r)$ is parametrized by a Zariski-open subset of $\Hilb_{P_{a_1,\ldots,a_r}(t)}(\bbP^n)$. The latter  
is smooth at $[Y] \in \CI_n(a_1,\ldots,a_r)$ and \cite[Theorem 4.3.5]{ser} yields
\[
T_{[Y]}\Hilb_{P_{a_1,\ldots,a_r}(t)}(\bbP^n) = H^0(N_{Y/\bbP^n}).
 \]
So $\dim \CI_n(a_1,\ldots,a_r) = h^0(N_{Y/\bbP^n})$, i.e., the dimension of the space of global sections of the normal bundle of $Y$.
\remarkend
\end{remark}
\end{samepage}

\begin{proposition}\label{prop:CI}
We have
\[
\dim \CI_n(a_1,\ldots,a_r) = \sum_{i=1}^r \coeff_{a_i} \left(\frac{\prod_{i=1}^r (1-t^{a_i})}{(1-t)^{n+1}}\right).
\]
\end{proposition}
\begin{proof}
Let $Y \in \CI_n(a_1,\ldots,a_r)$ be a general point. By Remark \ref{rmk:deformation}, $\dim\CI_n(a_1,\ldots,a_r) = h^0(N_{Y/\bbP^n})$. 
Since $Y$ is a complete intersection, its normal bundle is $N_{Y/\bbP^n} = \bigoplus_{i=1}^r \calO_{Y}(a_i)$.
Hence, the statement follows from the following equality: 
\[
h^0(\calO_Y(a_i)) = \coeff_{a_i}\left(\frac{\prod_{i=1}^r (1-t^{a_i})}{(1-t)^{n+1}}\right).
\]
To see that this equality holds, first notice that $Y$ is projectively normal \cite[Exercise~II.8.4]{h}, because it is a smooth complete intersection, by the generality assumption. 
So, for all $k\geq 0$, the restriction map $H^0(\calO_{\bbP^n}(k)) \rightarrow H^0(\calO_Y(k))$ is surjective. From the long exact sequence in cohomology of the short exact sequence \[0 \rightarrow \calI_Y(k) \rightarrow \calO_{\bbP^n}(k) \rightarrow \calO_Y(k) \rightarrow 0,\] one has $h^1(\calI_Y(k)) = 0$ for all $k\geq 0$.  Since $\HF_{\calS/I_Y}(d) = \coeff_d\left(\frac{\prod_{i=1}^r (1-t^{a_i})}{(1-t)^{n+1}}\right)$, where $I_Y$ is the homogeneous ideal of $Y$, the claimed equality follows.
\end{proof}

\begin{lemma}\label{lemma:equality}
For integers $e\geq0$ and $b_1,\ldots,b_s\geq1$, we have the following identity:
\[
\coeff_e\left(\frac{\prod_{i=1}^s (1-t^{b_i})}{(1-t)^{n+1}}\right) = \sum_{I \subseteq \{1,\ldots,s\}}(-1)^{\#I} \binom{n+e-\sum_{i \in I} b_i}{n}.
\]
Here $\binom{a}{b} = 0$ whenever $a < b$. 
\end{lemma}
\begin{proof}
Left to the reader.
\end{proof}

\begin{theorem}\label{thm:upper_bound}
Let $r<n$ and $a_1,\ldots,a_r\leq d/2$ be positive integers and take $\ell_{d/2} := \#\{i \mid a_i = d/2\}$. Then
\begin{align*}\label{eq:upper_bound}
\dim J_{a_1,\ldots,a_r} &\leq \binom{n+d}{d} - \coeff_d\left(\frac{\prod_{i=1}^r(1-t^{a_i})(1-t^{d-a_i})}{(1-t)^{n+1}}\right) + \binom{\ell_{d/2}}{2} - 1.
\end{align*}
When $d\geq3$, $a_1=\ldots=a_r=1$ and $r<\slrk_{d,n}^{\circ}$, equality holds.
\end{theorem}
\begin{proof}
First, we consider the case where $d\geq3$, $a_1=\ldots=a_r=1$. In this case, by \eqref{eq:tangent_join}, it is enough to compute the codimension of $(\ell_1,\ldots,\ell_r,g_1,\ldots,g_r)_d$ which corresponds to 
\[
\dim \calS_d/(\ell_1,\ldots,\ell_r,g_1,\ldots,g_r)_d = \dim \calS'_d/(\bar{g_1},\ldots,\bar{g_r})_d
\]
where $\calS' \cong \calS/(\ell_1,\ldots,\ell_r)$ is a polynomial ring in $n+1-r$ variables and $\bar{g_i}$ is the class of $g_i$ in $\calS'$. Since the $g_i$ are general of degree $d-1$, the latter dimension is obtained by \cite[Theorem 1]{hl} which states that
\[
\codim_{\PP\calS_d} J_{a_1,\ldots,a_r} = \coeff_d\left(\frac{(1-t^{d-1})^r}{(1-t)^{n+1-r}}\right).
\]
For the first statement, by Lemma \ref{lemma:bound1} and Proposition \ref{prop:CI}, it is enough to prove that
\begin{align*}
\sum_{j=1}^r \coeff_{a_j}  \left(\frac{\prod_{i=1}^r (1-t^{a_i})}{(1-t)^{n+1}}\right) &+ \binom{n+d}{d} -\coeff_{d} \left(\frac{\prod_{i=1}^r (1-t^{a_i})}{(1-t)^{n+1}}\right) - 1 = 
\\ & \binom{n+d}{d} - \coeff_d\left(\frac{\prod_{i=1}^r(1-t^{a_i})(1-t^{d-a_i})}{(1-t)^{n+1}}\right) + \binom{\ell_{d/2}}{2} - 1
\end{align*}
or, equivalently, to prove that
\begin{equation}\label{eq:rhsVSlhs}
\coeff_d\left(\frac{\prod_{i=1}^r(1-t^{a_i})(1-t^{d-a_i})}{(1-t)^{n+1}}\right) = \coeff_{d} \left(\frac{\prod_{i=1}^r (1-t^{a_i})}{(1-t)^{n+1}}\right) -\sum_{j=1}^r \coeff_{a_j} \left(\frac{\prod_{i=1}^r (1-t^{a_i})}{(1-t)^{n+1}}\right) + \binom{\ell_{d/2}}{2}.
\end{equation}
We analyze both sides of this equality. For the left hand side, we use Lemma \ref{lemma:equality} with $e = d$, $s = 2r$ and $(b_i,b_{r+i}) = (a_i,d-a_i)$ for $i = 1,\ldots,r$.
Since $a_i \leq d/2$ for all $i$, the summand corresponding to subset $I \subseteq \{1,\ldots,2r\}$ is zero whenever the intersection $I\cap\{r+1,\ldots,2r\}$ has more than two elements. The remaining summands correspond to subsets $I$ such that $I\subseteq\{1,\ldots,r\}$, $I=I'\cup \{r+j\}$ for $I'\subseteq\{1,\ldots,r\}$ and $j\in\{1,\ldots,r\}$ or $I=I'\cup\{r+j,r+k\}$ for $I'\subseteq\{1,\ldots,r\}$ and distinct $j,k\in\{1,\ldots,r\}$. In the last case, the summand is zero unless $a_j=a_k=d/2$ and $I'=\emptyset$. So we get
\[
\sum_{I \subseteq \{1,\ldots,r\}}(-1)^{\#I}\binom{n+d-\sum_{i \in I}a_i}{n} + \sum_{j=1}^r \sum_{I' \subseteq \{1,\ldots,r\}}(-1)^{\#I'+1}\binom{n+a_j-\sum_{i \in I'}a_i}{n} +  \binom{\ell_{d/2}}{2}.
\]
For the right hand side of \eqref{eq:rhsVSlhs}, we use Lemma \ref{lemma:equality} with $s = r$ and $b_i = a_i$ for $i = 1,\ldots,r$ and varying $e$. We get
\[
\sum_{I \subseteq \{1,\ldots,r\}}(-1)^{\#I}\binom{n+d-\sum_{i \in I}a_i}{n} - \sum_{j=1}^r \sum_{I \subseteq \{1,\ldots,r\}}(-1)^{\#I}\binom{n+a_j-\sum_{i \in I}a_i}{n} +  \binom{\ell_{d/2}}{2}.
\]
Hence \eqref{eq:rhsVSlhs} holds.
\end{proof}\vspace{-15pt}

\section{Numerical computations}\label{sec:proof}

Fix an integer $d\geq 5$. Let $n,r\geq 1$ and $1\leq a_1,\ldots,a_r\leq d/2$ be integers such that $r<\slrk_{d,n}^{\circ}$. Our goal is to prove that\vspace{-2pt}
\[
\dim J_{a_1,\ldots,a_r} \leq \dim \sigma_r (X_1)
\]
holds, and that we have equality if and only if $a_1=\ldots=a_r=1$. Write $\ell_j:=\#\{i\in\{1,\ldots,r\}\mid a_i=j\}$ for all $j\in\bbR$. By Theorem~\ref{thm:upper_bound}, it suffices to prove that, for fixed $n,r$, the value of
\[
F(a_1,\ldots,a_r):=\coeff_d\left(\frac{\prod_{i=1}^r(1-t^{a_i})(1-t^{d-a_i})}{(1-t)^{n+1}}\right) - \binom{\ell_{d/2}}{2}=\coeff_d\left(\frac{\prod_{i=1}^r(1-t^{a_i})}{(1-t)^{n+1}}\left(1-\sum_{i=1}^r t^{d-a_i}\right)\right)
\]
is minimal exactly when $a_1=\ldots=a_r=1$. We first prove that $F(a_1,\ldots,a_r)$ goes down when replacing all $a_i>2$ by~$2$. Afterwards, we deal with the cases where $a_1,\ldots,a_r\in\{1,2\}$. Take $\theta:=\max\{a_1,\ldots,a_r\}\leq d/2$. \phantom{test}\vspace*{-10pt}

\subsection{The case \texorpdfstring{$\theta>2$}{theta>2}} 
Write $P_k:=1+t+\ldots+t^k$ for $k\geq0$ and $P_{\infty}:=1/(1-t)$. 

\begin{lemma}\label{lem:increasing}
Let $s,\ell,k_1,\ldots,k_s\geq 0$ be integers. Then the coefficients of the power series
\[
P_{\infty}^{\ell+1} P_{k_1}\cdots P_{k_s}\vspace{-5pt}
\]
form a weakly increasing series.
\end{lemma}\vspace{-15pt}
\begin{proof}
We have\vspace{-5pt}
\[
P_{\infty}^{\ell+1}=\sum_{k=0}^\infty\binom{\ell+k}{k}t^k
\]
and so the lemma holds when $s=0$. When $f$ is a series whose coefficients increase weakly and $k\geq 0$ is an integer, then the same holds for the series $fP_k$. Hence the lemma holds for all $s$ using induction.
\end{proof}

We will often apply the next lemma with $g=P_a$ and $h=P_b$, where $a\geq b\geq0$ are integers.

\begin{lemma}
Let $f,g,h$ be series whose coefficients are all nonnegative and suppose that $\coeff_k(g)\geq\coeff_k(h)$ for all $k\geq0$. Then $\coeff_k(fg)\geq\coeff_k(fh)$ for all $k\geq0$.
\end{lemma}

\begin{theorem}\label{thm:theta>2}
Assume that $a_r=\theta>2$. Then $F(a_1,\ldots,a_r)>F(a_1,\ldots,a_{r-1},a_r-1)$.
\end{theorem}
\begin{proof}
Take
\[
f := \frac{\prod_{i=1}^{r-1}(1-t^{a_i})}{(1-t)^n}.
\]
Then we have
\[
F(a_1,\ldots,a_r)=\coeff_d\left(\frac{\prod_{i=1}^r(1-t^{a_i})}{(1-t)^{n+1}}\left(1-\sum_{i=1}^r t^{d-a_i}\right)\right)=\coeff_d\left(fP_{\theta-1}\left(1-\sum_{i=1}^r t^{d-a_i}\right)\right)
\]
and similarly
\[
F(a_1,\ldots,a_{r-1},a_r-1)=\coeff_d\left(fP_{\theta-2}\left(\left(1-\sum_{i=1}^r t^{d-a_i}\right)+t^{d-\theta}(1-t)\right)\right).
\]
We need to show that the difference
\[
\coeff_d\left(fP_{\theta-1}\left(1-\sum_{i=1}^r t^{d-a_i}\right)\right)-\coeff_d\left(fP_{\theta-2}\left(\left(1-\sum_{i=1}^r t^{d-a_i}\right)+t^{d-\theta}(1-t)\right)\right)
\]
is positive. 
This difference equals
\begin{align*}
\coeff_d\left(f\left(t^{\theta-1}\left(1-\sum_{i=1}^rt^{d-a_i}\right)-t^{d-\theta}+t^{d-1}\right)\right)&=\coeff_{d-\theta+1}(f)-\ell_{\theta-1}-(\ell_{\theta}-1)\coeff_1(f)-\coeff_{\theta}(f)\\
&=\coeff_{d-\theta+1}(f(1-t^{d-2\theta+1}))-\ell_{\theta-1}-(\ell_{\theta}-1)(n-\ell_1).
\end{align*}
Take
\[
g:=P_{\infty}^{n-r}P_{d-2\theta}\prod_{i=1}^{r-1}P_{a_i-1} =\frac{\prod_{i=1}^{r-1}(1-t^{a_i})}{(1-t)^n}(1-t^{d-2\theta+1})=f(1-t^{d-2\theta+1}).
\]
By Lemma \ref{lem:increasing}, the coefficients of $g$ are weakly increasing. So 
\[
\coeff_{d-\theta+1}(g)\geq \coeff_{\theta+1}(g).
\]
Write $m=n-\ell_1$. As $\ell_1+\ldots+\ell_{\theta}=r<\slrk_{d,n}^\circ\leq n$, we have $m>\ell_2+\ldots+\ell_{\theta}$. Note that
\begin{align*}
\coeff_{\theta+1}(g)&\geq \coeff_{\theta+1}\left(P_{\infty}^{n-r}P_{d-2\theta}P_1^{r-\ell_1-\ell_\theta}P_{\theta-1}^{\ell_\theta-1}\right)\\
&\geq \coeff_{\theta+1}\left(P_{\infty}P_1^{m-\ell_\theta-1}P_{\theta-1}^{\ell_\theta-1}\right)\\
&=\coeff_{\theta+1}\left(P_{\infty}^{\ell_\theta}P_1^{m-\ell_\theta-1}(1-t^\theta)^{\ell_\theta-1}\right)\\
&=\coeff_{\theta+1}\left(P_{\infty}^{\ell_\theta}P_1^{m-\ell_\theta-1}\right)-(\ell_\theta-1)(m-1)\\
&\geq \coeff_4\left(P_{\infty}^{\ell_\theta}P_1^{m-\ell_\theta-1}\right)-(\ell_\theta-1)(m-1)
\end{align*}
So it suffices to prove that
\begin{equation}\label{eq:prooftheta>2}
\coeff_4\left(P_{\infty}^{\ell_\theta}P_1^{m-\ell_\theta-1}\right)> \ell_{\theta-1}+(\ell_\theta-1)(2m-1)
\end{equation}
for all $\ell_{\theta-1}\geq0$, $\ell_{\theta}\geq1$ and $m>\ell_{\theta-1}+\ell_{\theta}$. Note that
\[
\ell_{\theta-1}+(\ell_\theta-1)(2m-1)\leq (m-\ell_\theta-1)+(\ell_\theta-1)(2m-1)=2\ell_\theta(m-1)-m.
\]
We have 
\[
\coeff_4\left(P_{\infty}^{\ell_\theta}P_1^{m-\ell_\theta-1}\right)\geq \coeff_4\left(P_{\infty}P_1^{m-2}\right)=\sum_{k=0}^4\binom{m-2}{k}
\]
which is strictly greater than $2(m-1)(m-1)-m\geq 2\ell_\theta(m-1)-m$ for $m\geq 10$. This leaves the case $m\leq 9$, where we verified that \eqref{eq:prooftheta>2} holds by computer.
This finishes the proof.
\end{proof}

By Theorem~\ref{thm:theta>2}, it suffices to focus on the cases where $a_1,\ldots,a_r\in\{1,2\}$. In these cases, we will regard $F(a_1,\ldots,a_r)$ as a function $A_{\ell_1,\ell_2}$ (defined below) depending only on $\ell_1$ and $\ell_2$.

\subsection{The case  \texorpdfstring{$\theta=2$}{theta=2}}

Recall that $d\geq5$. We define
\begin{align*}
A_{\ell_1,\ell_2}&:=\coeff_d\left(\frac{(1-t)^{\ell_1}(1-t^2)^{\ell_2}}{(1-t)^{n+1}}\left(1-\ell_1t^{d-1}-\ell_2t^{d-2}\right)\right)\mbox{ for $\ell_1,\ell_2\geq0$,}\\
B_{\ell_1,\ell_2}&:= A_{\ell_1-1,\ell_2+1}-A_{\ell_1,\ell_2}\mbox{ for $\ell_1\geq 1$ and $\ell_2\geq0$,}\\
C_{\ell_1,\ell_2}&:= B_{\ell_1-1,\ell_2+1}-B_{\ell_1,\ell_2}\mbox{ for $\ell_1\geq 2$ and $\ell_2\geq0$,}\\
D_{\ell_1,\ell_2}&:= C_{\ell_1-1,\ell_2+1}-C_{\ell_1,\ell_2}\mbox{ for $\ell_1\geq 3$ and $\ell_2\geq0$ and}\\
E_{\ell_1,\ell_2}&:= D_{\ell_1-1,\ell_2+1}-D_{\ell_1,\ell_2}\mbox{ for $\ell_1\geq 4$ and $\ell_2\geq0$.}
\end{align*}
The goal of this subsection is to prove the following theorem.

\begin{theorem}\label{thm:12->11}
We have $A_{\ell_1,\ell_2}>A_{\ell_1+\ell_2,0}$ for all integers $\ell_1\geq0$ and $\ell_2\geq1$ such that $\ell_1+\ell_2<\slrk_{d,n}^{\circ}$. 
\end{theorem}

We write $m=n-\ell_1$ and we assume that $\ell_1+\ell_2<n$. So $\ell_2<m$. In particular, we have $m\geq 1$.

\begin{lemma}
Let $\ell_1,\ell_2\geq0$ be integers such that $\ell_1+\ell_2<n$.
\begin{itemize}
\item[(a)] We have
\[
A_{\ell_1,\ell_2}=\coeff_d(P_{\infty}^{m+1-\ell_2}P_1^{\ell_2})-\ell_2\binom{m+2}{2}-\ell_1(m+1)+\ell_2^2.
\]
\item[(b)] When $\ell_1\geq1$, we have
\[
B_{\ell_1,\ell_2}=\coeff_{d-1}(P_{\infty}^{m+1-\ell_2}P_1^{\ell_2})-\binom{m+2}{2}-\ell_2m-\ell_1+1.
\]
\item[(c)] When $\ell_1\geq2$, we have
\[
C_{\ell_1,\ell_2}=\coeff_{d-2}(P_{\infty}^{m+1-\ell_2}P_1^{\ell_2})-2(m+1)-\ell_2.
\]
\item[(d)] When $\ell_1\geq3$, we have
\[
D_{\ell_1,\ell_2}=\coeff_{d-3}(P_{\infty}^{m+1-\ell_2}P_1^{\ell_2})-3.
\]
\item[(e)] When $\ell_1\geq4$, we have
\[
E_{\ell_1,\ell_2}=\coeff_{d-4}(P_{\infty}^{m+1-\ell_2}P_1^{\ell_2}).
\]
\end{itemize}
\end{lemma}
\begin{proof}
These calculations are straightforward.
\end{proof}

\begin{lemma}\label{lm:edcba}
Let $\ell_1\geq 1$ and $\ell_2\geq0$ be integers such that $\ell_1+\ell_2<n$.
\begin{itemize}
\item[(a)] When $\ell_1<\slrk_{d,n}^{\circ}$, we have $B_{\ell_1,0}>0$.
\item[(b)] When $\ell_1\geq 2$, we have $C_{\ell_1,\ell_2}\geq 0$.
\item[(c)] When $\ell_1\geq 3$, we have $D_{\ell_1,\ell_2}\geq 0$.
\item[(d)] When $\ell_1\geq 4$, we have $E_{\ell_1,\ell_2}\geq 2$.
\end{itemize}
\end{lemma}
\begin{proof}
We prove the parts of the lemma in reverse order.

(d).
We have 
$
E_{\ell_1,\ell_2}=\coeff_{d-4}(P_{\infty}^{m+1-\ell_2}P_1^{\ell_2})\geq \coeff_1(P_1^{m+1})=m+1\geq 2.
$

(c).
By (d), we have $D_{\ell_1,\ell_2}\geq D_{\ell_1+\ell_2,0}$. So we may assume that $\ell_2=0$. Now, we have
\[
D_{\ell_1,0}=\coeff_{d-3}(P_{\infty}^{m+1})-3=\binom{m+d-3}{d-3}-3\geq \binom{1+d-3}{d-3}-3=(d-2)-3\geq 0.
\]
(b).
By (c), we have $C_{\ell_1,\ell_2}\geq C_{\ell_1+\ell_2,0}$. So we may assume that $\ell_2=0$. Now, we have
\[
(m+1)\geq 2 \mbox{ and }\frac{(m+d-2)\cdots(m+2)}{(d-2)!}-2\geq \frac{(1+d-2)\cdots(1+2)}{(d-2)!}-2=\frac{d-1}{2}-2\geq0
\]
and so\vspace{-2pt}
\[
C_{\ell_1,0}=\coeff_{d-2}(P_{\infty}^{m+1})-2(m+1)=\binom{m+d-2}{d-2}-2(m+1)=(m+1)\left(\frac{(m+d-2)\cdots(m+2)}{(d-2)!}-2\right)\geq0.
\]
(a).
By (b), $B_{\ell_1,\ell_2}\geq B_{\ell_1+\ell_2,0}$. So we may assume $\ell_2 = 0$. Since $\ell_1<\slrk_{d,n}^{\circ}$, we have $\ell_1(m+1)<\binom{m+d}{d}$. So $d!\ell_1<(m+d)\cdots(m+2)$. We get
\begin{align*}
d!B_{\ell_1,0}&=d!\left(\coeff_{d-1}(P_{\infty}^{m+1})-\binom{m+2}{2}-\ell_1+1\right)\\
&=d!\left(\binom{m+d-1}{d-1}-\frac{m(m+3)}{2}\right)-d!\ell_1\\
&>d!\left(\binom{m+d-1}{d-1}-\frac{m(m+3)}{2}\right)-(m+d)\cdots(m+2)\\
&=d(m+d-1)\cdots(m+1)-\frac{d!}{2}m(m+3)-(m+d)\cdots(m+2)\\
&=(m+d-1)\cdots(m+2)\left(d(m+1)-(m+d)\right)-\frac{d!}{2}m(m+3)\\
&=(m+d-1)\cdots(m+2)(d-1)m-\frac{d!}{2}m(m+3)\\
&=m\left((m+d-1)\cdots(m+2)(d-1)-\frac{d!}{2}(m+3)\right).
\end{align*}\vspace{-2pt}
So it suffices to prove that
\[
c_0+c_1m+\ldots+c_{d-2}m^{d-2}:=(m+d-1)\cdots(m+2)(d-1)-\frac{d!}{2}(m+3)\geq0
\]
We have\vspace{-2pt}
\begin{align*}
c_1&=(d-1)\coeff_1\left((m+d-1)\cdots(m+2)\right)-\frac{d!}{2}\\
&=(d-1)\sum_{i=2}^{d-1}\frac{(d-1)!}{i}-\frac{d!}{2}\\
&=(d-1)!\left(\sum_{i=2}^{d-1}\frac{d-1}{i}-\frac{d}{2}\right)\\
&\geq(d-1)!\left(\frac{d-1}{2}+\frac{d-1}{d-1}-\frac{d}{2}\right)>0
\end{align*}
and $c_i>0$ for $i=2,\ldots,d-3$. Hence
\begin{align*}
c_0+c_1m+\ldots+c_{d-2}m^{d-2}&\geq c_0+c_1+\ldots+c_{d-2}\\
&=(1+d-1)\cdots(1+2)(d-1)-\frac{d!}{2}(1+3)\\
&=\frac{d!}{2}(d-1)-\frac{d!}{2}\cdot 4=\frac{d!}{2}(d-5)\geq0.
\end{align*}\vspace{-2pt}
This finishes the proof.
\end{proof}

Theorem~\ref{thm:12->11} now follows easily.

\begin{proof}[Proof of Theorem~\ref{thm:12->11}]
By parts (a) and (b) of Lemma~\ref{lm:edcba}, we have
\[
A_{\ell_1,\ell_2}-A_{\ell_1+1,\ell_2-1}=B_{\ell_1+1,\ell_2-1}\geq B_{\ell_1+\ell_2,0}>0.
\]
Repeating this, we find that 
\[
A_{\ell_1,\ell_2}>A_{\ell_1+1,\ell_2-1}>\cdots>A_{\ell_1+\ell_2,0}
\]
as desired.
\end{proof}

\subsection{The conclusion of the proof}

\begin{proof}[Proof of Theorem~\ref{thm:main}]
Let $d\geq 5$, $n,r\geq 1$ and $1\leq a_1,\ldots,a_r\leq d/2$ be integers such that $r<\slrk_{d,n}^{\circ}$. We need to show that
\[
\dim J_{a_1,\ldots,a_r} \leq \dim \sigma_r (X_1)
\]
holds, and that we have equality only for $a_1=\ldots=a_r=1$. By Theorem~\ref{thm:upper_bound}, it suffices to prove that 
\[
F(a_1,\ldots,a_r)
\]
is minimal exactly when $a_1=\ldots=a_r=1$. By Theorem~\ref{thm:theta>2}, it suffices to do this in the case where $a_1,\ldots,a_r\in\{1,2\}$. Here, we have $F(a_1,\ldots,a_r)=A_{\ell_1,\ell_2}$ and so the statement holds by Theorem~\ref{thm:12->11}.
\end{proof}

\end{document}